\UseRawInputEncoding

\documentclass[oneside, 11pt]{amsart} 

\usepackage{amsmath,amsthm,amssymb,epic}
\usepackage{eqlist,eqparbox}
\usepackage{amsfonts}
\usepackage{latexsym}
\usepackage[2emode]{psfrag}
\usepackage{amsthm}
\usepackage{amsmath}
\usepackage[all]{xy}

\newcommand{\Title}[1]{\bigskip\bigskip\centerline{\bf #1}\bigskip}
\newcommand{\Author}[1]{\medskip\centerline{ \it #1}}

\newcommand{\Affiliation}[1]{\medskip\centerline{#1}}
\newcommand{\Email}[1]{\medskip\centerline{#1}\bigskip}

\begin{document}

\newcommand{\N}{\mbox {$\mathbb N $}}
\newcommand{\Z}{\mbox {$\mathbb Z $}}
\newcommand{\Q}{\mbox {$\mathbb Q $}}
\newcommand{\R}{\mbox {$\mathbb R $}}
\newcommand{\lo }{\longrightarrow }
\newcommand{\ul}{\underleftarrow }
\newcommand{\rl}{\underrightarrow }
\newcommand{\rs }{\rightsquigarrow }
\newcommand{\ra }{\rightarrow }
\newcommand{\dd }{\rightsquigarrow }
\newcommand{\ol }{\overline }
\newcommand{\la }{\langle }
\newcommand{\tr }{\triangle }
\newcommand{\xr }{\xrightarrow }
\newcommand{\de }{\delta }
\newcommand{\pa }{\partial }
\newcommand{\LR }{\Longleftrightarrow }
\newcommand{\Ri }{\Rightarrow }
\newcommand{\va }{\varphi }
\newcommand{\Den}{{\rm Den}\,}
\newcommand{\Ker}{{\rm Ker}\,}
\newcommand{\Reg}{{\rm Reg}\,}
\newcommand{\Fix}{{\rm Fix}\,}
\newcommand{\Sup}{{\rm Sup}\,}
\newcommand{\Inf}{{\rm Inf}\,}
\newcommand{\Img}{{\rm Im}\,}
\newcommand{\Id}{{\rm Id}\,}

\newtheorem{theorem}{Theorem}[section]
\newtheorem{lemma}[theorem]{Lemma}
\newtheorem{proposition}[theorem]{Proposition}
\newtheorem{corollary}[theorem]{Corollary}
\newtheorem{definition}[theorem]{Definition}
\newtheorem{example}[theorem]{Example}
\newtheorem{examples}[theorem]{Examples}
\newtheorem{xca}[theorem]{Exercise}
\theoremstyle{remark}
\newtheorem{remark}[theorem]{Remark}
\newtheorem{remarks}[theorem]{Remarks}
\numberwithin{equation}{section}

\def\leftmark{L.C. Ciungu}

\Title{QUANTUM-WAJSBERG ALGEBRAS} 
\title[Quantum-Wajsberg algebras]{}
                                                                           
\Author{\textbf{LAVINIA CORINA CIUNGU}}
\Affiliation{Department of Mathematics} 
\Affiliation{St Francis College}
\Affiliation{179 Livingston Street, Brooklyn, NY 11201, USA}
\Email{lciungu@sfc.edu}

\begin{abstract} 
Starting from involutive BE algebras, we redefine the quantum-MV algebras, by introducing and studying the 
notion of quantum-Wajsberg algebras. We define the $\vee$-commutative quantum-Wajsberg algebras and we investigate their properties. We also prove that any Wajsberg algebra is a quantum-Wajsberg algebra, 
and give conditions for quantum-Wajsberg algebras to be Wajsberg algebras. 
Furthermore, we prove that Wajsberg algebras are both quantum-Wajsberg algebras and commutative quantum-B algebras. 
Finally, we provide certain conditions for quantum-MV algebras to be MV algebras. \\

\noindent
\textbf{Keywords:} {quantum-Wajsberg algebra, quantum-MV algebra, MV algebra, Wajsberg algebra, BE algebra, 
m-BE algebra, quantum-B algebra} \\
\textbf{AMS classification (2020):} 06D35, 06F35, 03G25, 06A06
\end{abstract}

\maketitle

\section{Introduction} 

In the last decades, developing algebras related to the logical foundations of quantum mechanics became a 
central topic of research. 
Generally known as quantum structures, these algebras serve as models for the formalism of quantum mechanics. 
The behavior of physical systems which describe the events concerning elementary particles is different from 
that of physical systems observed in classical physics. 
Since with the classical propositional calculus these events cannot be described, the logic of quantum mechanics 
becomes an important tool for evaluating the propositions and propositional formulas on these systems. 
The many-valued logic descovered by \L ukasiewicz became the foundation for fuzzy logic, and the class of quantum structures has been growing: orthomodular lattices, orthomodular posets,  orthoalgebras, effect algebras, D-posets, 
MV algebras, BL algebras, hoops, BCK/BCI algebras. 
When the operations of addition and multiplication were no longer assumed to be commutative or there are two implications, new quantum structures appeared: pseudo-MV/BL/BCK/BCI algebras, pseudo-hoops, residuated lattices, 
pseudo-effect algebras. 
More detailed considerations on this topic are presented in the monograph \cite{DvPu} by A. Dvure\v censkij and S. Pulmannov\'a, as well as in the monograph \cite{Bot1} by M. Botur, I. Chajda, R. Hala\v s, J. $\rm K \ddot{u}hr$ 
and J. Paseka. 
These quantum structures have been generalized in two directions: \\
{\bf I.} Quantum-MV algebras were introduced by R. Giuntini in \cite{Giunt1} as non-lattice generalizations of 
MV algebras (\cite{Chang, Cig1}) and as non-idempotent generalizations of orthomodular lattices. 
Quantum-MV algebras also generalize the effect algebras (\cite{Foulis, DvVe1}) and the orthomodular lattices (\cite{Kalm, Ptak}).   
These structures were intensively studied by R. Giuntini (\cite{Giunt2, Giunt3, Giunt4, Giunt5, Giunt6}), 
A. Dvure\v censkij and S. Pulmannov\'a (\cite{DvPu}), R. Giuntini and S. Pulmannov\'a (\cite{Giunt7}) and by 
A. Iorgulescu in \cite{Ior30, Ior31, Ior32, Ior33, Ior34, Ior35}. \\
{\bf II.} Quantum-B algebras defined and investigated by W. Rump and Y.C. Yang (\cite{Rump2, Rump1}) 
arise from the concept of quantales, which was introduced in 1984 as a framework for quantum mechanics 
with a view toward non-commutative logic (\cite{Mulv1}). 
Many implicational algebras studied so far (pseudo-effect algebras, residuated lattices, pseudo-MV/BL/MTL algebras, bounded non-commutative R$\ell$-monoids, pseudo-hoops, pseudo-BCK/BCI algebras) are \\ 
quantum-B algebras. 
Interesting results on quantum-B algebras have been presented in \cite{Rump3, Rump4, Han1, Han2}. \\
Unfortunately, there are no connections between quantum-MV algebras and quantum-B algebras; there are only particular 
algebras, such as MV algebras, that belong to both structures. 
To be able to compare certain quantum structures, it is preferable that they have the same 
signature, which would also allow possible generalizations of the quantum structures. \\ 
In this paper, we introduce the notion of quantum-Wajsberg algebras, by redefining the quantum-MV algebras starting 
from involutive BE algebras. 
We give a characterization of quantum-Wajsberg algebras, investigate their properties, and show that, 
in general, quantum-Wajsberg algebras are not (commutative) quantum-B algebras. 
We also define the $\vee$-commutative quantum-Wajsberg algebras and study their properties. 
Furthermore, we prove that any Wajsberg algebra (bounded $\vee$-commutative BCK algebra) is a quantum-Wajsberg algebra, 
and we give a condition for a quantum-Wajsberg algebra to be a Wajsberg algebra. 
We prove that Wajsberg algebras are both quantum-Wajsberg algebras and commutative quantum-B algebras. 
We establish the connection between quantum-Wajsberg algebras and quantum-MV algebras, proving that the quantum-Wajsberg algebras are term-equivalent to quantum-MV algebras. 
We show that, in general, the quantum-Wajsberg algebras are not commutative quantum-B algebras and,  
if a quantum-Wajsberg algebra is self distributive, then the corresponding quantum-MV algebra is an MV algebra.

$\vspace*{5mm}$

\section{Preliminaries}

\indent
In this section, we recall some basic notions and results regarding quantum-MV algebras, quantum-B algebras, 
BCK algebras, and Wajsberg algebras used in the paper. 

Starting from the systems of positive implicational calculus, weak systems of positive implicational calculus 
and BCK systems, in 1966 Y. Imai and K. Is\`eki introduced the BCK algebras (\cite{Imai}) as right algebras. 
BCK algebras are also used in a dual form, with an implication $\ra$ and with one constant element $1$, 
that is the greatest element (\cite{Kim2}). 
An equivalent shorter definition of BCK algebras is the following.

\begin{definition} \label{pr-50} $\rm($\cite{Grza}$\rm)$ 
A \emph{(left-)BCK algebra} is an algebra $(X,\ra,1)$ of type $(2,0)$ satisfying the following conditions, 
for all $x,y,z\in X$: \\
$(BCK_1)$ $(x\ra y)\ra ((y\ra z)\ra (x\ra z))=1;$ \\
$(BCK_2)$ $1\ra x=x;$ \\
$(BCK_3)$ $x\ra 1=1;$ \\ 
$(BCK_4)$ $x\ra y=1$ and $y\ra x=1$ imply $x=y$. 
\end{definition}

In this paper, by BCK algebras we understand the left-BCK algebras. 
If $(X,\ra,1)$ is a BCK algebra, for $x,y\in X$ we define the relation $\le$ by $x\le y$ if and only if $x\ra y=1$, 
and $\le$ is a partial order on $X$. \\
We also define $x\vee y=(x\ra y)\ra y$. 
If the BCK algebra satisfies axiom \\
$($$\vee$-$comm)$ $x\vee y=y\vee x$, for all $x,y$, \\ 
we shall say that it is \emph{$\vee$-commutative}.  

\begin{proposition} \label{pr-60} $\rm($\cite{Iseki}$\rm)$ Let $(X,\ra,1)$ be a BCK algebra. 
Then, the following hold, for all $x,y,z\in X$: \\
$(1)$ $x\le y$ implies $z\ra x\le z\ra y$ and $y\ra z\le x\ra z;$ \\
$(2)$ $x\le y\ra x;$ \\
$(3)$ $x\ra (y\ra z)=y\ra (x\ra z);$ \\
$(4)$ $((y\ra x)\ra x)\ra x=y\ra x;$ \\
$(5)$ $z\le y\ra x$ iff $y\le z\ra x$.  
\end{proposition}
 
Wajsberg algebras were introduced in 1984 by Font, Rodriguez and Torrens in \cite{Font1} as algebraic model 
of  $\aleph_0$-valued \L ukasiewicz logic.   

\begin{definition} \label{pr-70} $\rm($\cite{Font1}$\rm)$ 
A \emph{(left-)Wajsberg algebra} is an algebra $(X,\ra,^*,1)$ of type $(2,1,0)$ satisfying the following conditions,  
for all $x,y,z\in X$:\\
$(W_1)$ $1\ra x=x;$ \\
$(W_2)$ $(y\ra z)\ra ((z\ra x)\ra (y\ra x))=1;$ \\
$(W_3)$ $(x\ra y)\ra y=(y\ra x)\ra x;$ \\
$(W_4)$ $(x^*\ra y^*)\ra (y\ra x)=1$. 
\end{definition}
\noindent
Note that $(W_3)=($$\vee$-$comm)$. \\ 
Wajsberg algebras are bounded with $0=1^*$, and they are involutive. 
In any Wajsberg algebra $(X,\ra,^*,1)$ the following hold, for all $x,y,z\in X$ (\cite{Font1}): \\
$(W_5)$ $x\ra x=1;$ \\
$(W_6)$ $x\ra 1=1;$ \\
$(W_7)$ $x\ra (y\ra z)=y\ra (x\ra z)$. \\
It was proved in \cite{Font1} that Wajsberg algebras are term-equivalent to MV algebras. \\

The concept of quantum-B algebras was introduced by W. Rump and Y.C. Yang in \cite{Rump2,Rump1}. 
The quantum-B algebras generalize many implicational algebras, such as pseudo-effect algebras, residuated lattices, pseudo-MV/BL/MTL algebras, bounded non-commutative R$\ell$-monoids, pseudo-hoops, pseudo-BCK/BCI algebras. 

\begin{definition} \label{pr-40} $\rm($\cite{Rump1},\cite{Rump2}$\rm)$ 
A \emph{quantum-B algebra} is a structure $(X,\le,\ra,\rs)$ such that $(X,\le)$ is a partially ordered set and 
$\ra, \rs$ are two binary operations satisfying the following axioms, for all $x, y, z\in X$: \\
$(QB_1)$ $y\ra z\le (x\ra y)\ra (x\ra z);$ \\
$(QB_2)$ $y\rs z\le (x\rs y)\rs (x\rs z);$ \\
$(QB_3)$ $y\le z$ implies $x\ra y\le x\ra z;$ \\
$(QB_4)$ $x\le y\ra z$ iff $y\le x\rs z$. 
\end{definition} 

\noindent
A quantum-B algebra $X$ is said to be \emph{unital} if there exists an element $u\in X$ such that 
$u\ra x=u\rs x=x$, for all $x\in X$. The element $u$ is called a \emph{unit element} and the unit element is 
unique (\cite{Rump1}). 
A quantum-B algebra $(X,\le,\ra,\rs)$ is called \emph{commutative} if $x\ra y=x\rs y$, for all $x,y\in X$. 
According to \cite[Prop. 1.3.19]{Ior35} and \cite[Cor. 2]{Rump2}, pseudo-BCK algebras are term-equivalent to 
integral quantum-B algebras (i.e. unital quantum-B algebras $(X,\le,\ra,\rs,1)$ with $1$ verifying $x\le 1$, 
for all $x\in X$), hence BCK algebras are term-equivalent to commutative integral quantum-B algebras. \\

The supplement algebras were introduced by S. Gudder in \cite{Gudder} and they are a common set of axioms for 
MV and quantum-MV algebras. 

\begin{definition} \label{pr-10} $\rm($\cite{Gudder}$\rm)$ 
A \emph{suplement algebra} (\emph{S-algebra, for short}) is an algebra $(X,\oplus,^*,0,1)$ of type $(2,1,0,0)$   
satisfying the following axioms, for all $x, y, z\in X$: \\
$(S_1)$ $x\oplus y=y\oplus x;$ \\
$(S_2)$ $x\oplus (y\oplus z)=(x\oplus y)\oplus z;$ \\
$(S_3)$ $x\oplus x^*=1;$ \\
$(S_4)$ $x\oplus 0=x;$ \\
$(S_5)$ $x^{**}=x;$ \\
$(S_6)$ $0^*=1;$ \\
$(S_7)$ $x\oplus 1=1$.
\end{definition}

The following additional operations can be defined in a supplement algebra: \\
$\hspace*{3cm}$ $x\odot y=(x^*\oplus y^*)^*$, \\
$\hspace*{3cm}$ $x\Cap_S y=(x\oplus y^*)\odot y$, \\
$\hspace*{3cm}$ $x\Cup_S y=(x\odot y^*)\oplus y$. \\
Obviously, $(x\Cup_S y)^*=x^*\Cap_S y^*$ and $(x\Cap_S y)^*=x^*\Cup_S y^*$. \\ 
For all $x,y\in X$, we define $\le_S$ by $x\le_S y$ iff $x=x\Cap_S y$.  

\begin{definition} \label{pr-20} $\rm($\cite{Giunt2}$\rm)$ 
An \emph{(right-)MV algebra} is an S-algebra $(X,\oplus,^*,0,1)$ satisfying the following axiom, for all $x, y\in X$: \\
$(MV)$ $(x^*\oplus y)^*\oplus y=(x\oplus y^*)^*\oplus x$. 
\end{definition}

\begin{remark} \label{pr-20-10} 
The condition (MV) is a so-called \emph{\L ukasiewicz axiom} (\cite{DvPu}) and it is equivalent to 
$x\Cap_S y=y\Cap_S x$, for all $x,y\in X$. \\
Indeed, $(x^*\oplus y)^*\oplus y=((x^*\oplus y)\odot y^*)^*$ and 
$(x\oplus y^*)^*\oplus x=(y^*\oplus x)^*\oplus x=((y^*\oplus x)\odot x^*)^*$.
It follows that $((x^*\oplus y)\odot y^*)^*=((y\oplus x^*)\odot x^*)^*$, and replacing $x$ by $x^*$ and $y$ by $y^*$ 
we get $((x\oplus y^*)\odot y)^*=((y^*\oplus x)\odot x)^*$. 
Hence $(x\oplus y^*)\odot y=(y^*\oplus x)\odot x$, that is $x\Cap_S y=y\Cap_S x$. \\
As a consequence, condition (MV) is also equivalent to $x\Cup_S y=y\Cup_S x$, for all $x,y\in X$. 
\end{remark}

The quantum-MV algebras were introduced by R. Giuntini as right algebras. 

\begin{definition} \label{pr-30} $\rm($\cite{Giunt1}$\rm)$ 
A \emph{(right-)quantum-MV algebra} (\emph{QMV algebra, for short}) is an S-algebra $(X,\oplus,^*,0,1)$ satisfying the 
following axiom, for all $x, y, z\in X$: \\
$(QMV)$ $x\oplus ((x^*\Cap_S y)\Cap_S (z\Cap_S x^*))=(x\oplus y)\Cap_S (x\oplus z)$. 
\end{definition}

In this paper we will use the notion of left-quantum-MV algebras. 

\begin{definition} \label{pr-30-20} $\rm($\cite{Ior30}$\rm)$
A \emph{left-m-BE algebra} is an algebra $(X,\odot,^{*},1)$ of type $(2,1,0)$ satisfying the following properties, 
for all $x,y,z\in X$: \\ 
(PU) $1\odot x=x=x\odot 1;$ \\
(Pcomm) $x\odot y=y\odot x;$ \\
(Pass) $x\odot (y\odot z)=(x\odot y)\odot z;$ \\ 
(m-L) $x\odot 0=0;$ \\ 
(m-Re) $x\odot x^{*}=0$, \\
where $0:=1^*$. 
\end{definition}

A left-m-BE algebra $X$ is involutive if it satisfies condition $(DN)$ $x^{**}=x$, for all $x\in X$. 
In an involutive left-m-BE algebra the following operations and relations are defined, for all $x,y\in X$: \\ 
$\hspace*{3cm}$ $x\oplus y=(x^*\odot y^*)^*$ (obviously, $x\odot y=(x^*\oplus y^*)^*$), \\
$\hspace*{3cm}$ $x\wedge_m^M y=(x^*\odot y)^*\odot y=y\odot (y\odot x^*)^*$, \\
$\hspace*{3cm}$ $x\vee_m^M y=(x^*\wedge_m^M y^*)^*$, \\
$\hspace*{3cm}$ $x\le_m^M y$ iff $x=x\wedge_m^M y$, \\ 
$\hspace*{3cm}$ $x\le_m y$ iff $x\odot y^*=0$. 
                
\begin{definition} \label{pr-30-30} $\rm($\cite[Def. 3.10]{Ior34}$\rm)$
A \emph{(left-)quantum-MV algebra}, or a \\
\emph{(left-)QMV algebra} for short, is an involutive (left-)m-BE algebra
$(X,\odot,^{*},1)$ verifying the following axiom: for all $x,y,z\in X$, \\
(Pqmv) $x\odot ((x^*\vee_m^M y)\vee_m^M (z\vee_m^M x^*))=(x\odot y)\vee_m^M (x\odot z)$. 
\end{definition}

\begin{remark} \label{pr-30-40}
A. Iorgulescu and M. Kinyon defined and studied in \cite{Ior34} the notion of \emph{right-m-BE algebras} as 
algebras $(X,\oplus,^{*},0)$ of type $(2,1,0)$ satisfying the following properties: 
(SU) $0\oplus x=x=x\oplus 0$; (Scomm) $x\oplus y=y\oplus x$; (Sass) $x\oplus (y\oplus z)=(x\oplus y)\oplus z$; 
(m-L$^{R}$) $x\oplus 1=1$; (m-Re$^{R}$) $x\oplus x^{*}=1$, where $1:=0^*$. \
A right-m-BE algebra $X$ is involutive if it satisfies condition $(DN)$.  
They noted that the involutive right-m-BE algebras are the S-algebras defined by S. Gudder. \\
In an involutive right-m-BE algebra the following operation is defined, for all $x,y\in X$: \\ 
$\hspace*{2cm}$ $x\odot y=(x^*\oplus y^*)^*$ (obviously, $x\oplus y=(x^*\odot y^*)^*$), \\
The right-quantum-MV algebras can be defined as involutive right-m-BE algebras satisfying the following axiom, 
for all $x,y,z\in X$: \\
(Sqmv) $x\oplus [(x^*\wedge_m^M y)\wedge_m^M (z\wedge_m^M x^*)]=(x\oplus y)\wedge_m^M (x\oplus z)$. \\
According to \cite[Prop. 3.11]{Ior34}, conditions (Sqmv) and (Pqmv) are equivalent. 
The left and right-quantum-MV algebras were intensively studied by A. Iorgulescu in 
\cite{Ior30, Ior31, Ior32, Ior33, Ior34, Ior35}. 
\end{remark}

\begin{remarks} \label{pr-30-50} 
$(1)$ A right-MV algebra is an involutive right-m-BE algebra $(X,\oplus,^*,0)$ satisfying condition $(MV)$. \\
$(2)$ A left-MV algebra is an involutive left-m-BE algebra $(X,\odot,^*,1)$ satisfying the following condition for all $x,y\in X$: \\
$(MV^{'})$ $(x^*\odot y)^*\odot y=(y^*\odot x)^*\odot x$. \\
We shall further work with the definition of (left-)MV algebras.  
\end{remarks}

$\vspace*{1mm}$

\section{Quantum-Wajsberg algebras}

In this section, we redefine the quantum-MV algebras starting from involutive BE algebras by introducing and studying  the notion of quantum-Wajsberg algebras. We give a characterization of quantum-Wajsberg algebras, investigate their properties, and show that, in general, the quantum-Wajsberg algebras are not (commutative) quantum-B algebras. 

The (left-)BE algebras were introduced in \cite{Kim1} as algebras $(X,\ra,1)$ of type $(2,0)$ satisfying the following conditions, for all $x,y,z\in X$: \\
$(BE_1)$ $x\ra x=1;$ \\
$(BE_2)$ $x\ra 1=1;$ \\
$(BE_3)$ $1\ra x=x;$ \\
$(BE_4)$ $x\ra (y\ra z)=y\ra (x\ra z)$. \\
A relation ``$\le$" is defined on $X$ by $x\le y$ iff $x\ra y=1$. 
A BE algebra $X$ is \emph{bounded} if there exists $0\in X$ such that $0\le x$, for all $x\in X$. 
In a bounded BE algebra $(X,\ra,0,1)$ we define $x^*=x\ra 0$, for all $x\in X$. 
A bounded BE algebra $X$ is called \emph{involutive} if $x^{**}=x$, for any $x\in X$. \\
Obviously, any (left-)BCK algebra is a (left-)BE algebra, but the exact connection between BE algebras and 
BCK algebras is made in the papers \cite{Ior16}: a BCK algebra is a BE algebra satisfying $(BCK_4)$ (antisymmetry) 
and $(BCK_1)$. 

\begin{remark} \label{qbe-05}
According to \cite[Cor. 17.1.3]{Ior35}, the involutive (left-)BE algebras $(X,\ra,^*,1)$ are term-equivalent to involutive (left-)m-BE algebras $(X,\odot,^*,1)$, by the mutually inverse transformations 
(\cite{Ior30, Ior35}): \\ 
$\hspace*{2cm}$ $\Phi:$\hspace*{0.2cm}$ x\odot y:=(x\ra y^*)^*$ $\hspace*{0.1cm}$ and  
                $\hspace*{0.1cm}$ $\Psi:$\hspace*{0.2cm}$ x\ra y:=(x\odot y^*)^*$. 
\end{remark}
\noindent
In what follows, by BE algebras we understand the left-BE algebras. 

\begin{lemma} \label{qbe-10} Let $(X,\ra,1)$ be a BE algebra. Then, the following hold, for all $x,y,z\in X$: \\
$(1)$ $x\ra (y\ra x)=1;$ \\
$(2)$ $x\le (x\ra y)\ra y$. \\
If $X$ is bounded, then: \\
$(3)$ $x\ra y^*=y\ra x^*;$ \\
$(4)$ $x\le x^{**}$. \\
If $X$ is involutive, then: \\
$(5)$ $x^*\ra y=y^*\ra x;$ \\
$(6)$ $x^*\ra y^*=y\ra x;$ \\
$(7)$ $(x^*\ra y)^*\ra z=x^*\ra (y^*\ra z)$.  
\end{lemma}
\begin{proof} 
$(1)$: By $(BE_4)$, $(BE_1)$ and $(BE_2)$, $x\ra (y\ra x)=y\ra (x\ra x)=y\ra 1=1$. \\
$(2)$: Similarly, $x\ra ((x\ra y)\ra y)=(x\ra y)\ra (x\ra y)=1$, hence $x\le (x\ra y)\ra y$. \\ 
$(3)$: It follows by $(BE_4)$ for $z:=0$. \\
$(4)$: It follows by $(2)$ for $y:=0$. \\
$(5)$: Replace $x$ by $x^*$ and $y$ by $y^*$ in $(3)$. \\
$(6)$: Replace $y$ by $y^*$ in $(5)$. \\
$(7)$: Using $(5)$ and $(BE_4)$, we get: \\
$\hspace*{2.00cm}$ $(x^*\ra y)^*\ra z=z^*\ra (x^*\ra y)=x^*\ra (z^*\ra y)$ \\
$\hspace*{4.20cm}$ $=x^*\ra (y^*\ra z)$. 
\end{proof}

\noindent
If a BE algebra $X$ is involutive, we define the operation: \\
$\hspace*{3cm}$ $x\wedge y=(x^*\vee y^*)^*=((x^*\ra y^*)\ra y^*)^*$, \\
and the relation $\le_Q$ by: \\
$\hspace*{3cm}$ $x\le_Q y$ iff $x=x\wedge y$. \\
One can easily check that, in the involutive m-BE algebras,  
$\wedge_m^M= \Cap_S$, $\vee_m^M=\Cup_S$ and $\le_m^M \Longleftrightarrow \le_S$. 
Moreover, in the involutive case, by $\Phi$ and $\Psi$, we have $x\wedge y=x\wedge_m^M y$, $x\vee y=x\vee_m^M y$ 
and $x\le y \Longleftrightarrow x\le_m y$ (\cite[Th. 17.1.1]{Ior35}), $x\le_Q y \Longleftrightarrow x\le_m^M y$. 

\begin{proposition} \label{qbe-20} Let $X$ be an involutive BE algebra. 
Then, the following hold, for all $x,y\in X$: \\
$(1)$ $x\le_Q y$ implies $x=y\wedge x$ and $y=x\vee y;$ \\
$(2)$ $\le_Q$ is reflexive and antisymmetric; \\
$(3)$ $x\vee y=(x^*\wedge y^*)^*;$ \\ 
$(4)$ $x\le_Q y$ implies $x\le y;$ \\
$(5)$ $0\le_Q x \le_Q 1;$ \\
$(6)$ $0\wedge x=x\wedge 0=0$ and $1\wedge x=x\wedge 1=x;$ \\
$(7)$ $x\wedge (y\wedge x)=y\wedge x$ and $x\wedge (x\wedge y)=x\wedge y$. 
\end{proposition}
\begin{proof}
$(1)$: If $x\le_Q y$, then $x=x\wedge y$, and by $(BE_4)$, we have: \\
$\hspace*{2cm}$ $y\wedge x=y\wedge (x\wedge y)=((y^*\ra (x\wedge y)^*)\ra (x\wedge y)^*)^*$ \\
$\hspace*{3cm}$ $=((y^*\ra ((x^*\ra y^*)\ra y^*))\ra (x\wedge y)^*)^*$ \\
$\hspace*{3cm}$ $=(((x^*\ra y^*)\ra (y^*\ra y^*))\ra (x\wedge y)^*)^*$ \\
$\hspace*{3cm}$ $=(((x^*\ra y^*)\ra 1)\ra (x\wedge y)^*)^*$ \\
$\hspace*{3cm}$ $=(1\ra (x\wedge y)^*)^*=(x\wedge y)^{**}=x\wedge y=x$. \\ 
Similarly we get: \\
$\hspace*{2cm}$ $x\vee y=(x\wedge y)\vee y=((x\wedge y)\ra y)\ra y$ \\
$\hspace*{3cm}$ $=(((x^*\ra y^*)\ra y^*)^*\ra y)\ra y$ \\
$\hspace*{3cm}$ $=(y^*\ra ((x^*\ra y^*)\ra y^*))\ra y$ \\
$\hspace*{3cm}$ $=((x^*\ra y^*)\ra (y^*\ra y^*))\ra y$ \\
$\hspace*{3cm}$ $=((x^*\ra y^*)\ra 1)\ra y=1\ra y=y$. \\ 
$(2)$: Since $x\wedge x=((x^*\ra x^*)\ra x^*)^*=(1\ra x^*)^*=x$, $\le_Q$ is reflexive. \\
If $x\le_Q y$ and $y\le_Q x$, we have $y=y\wedge x=x$, by $(1)$. Hence $\le_Q$ is antisymmetric. \\
$(3)$: It is obvious. \\
$(4)$: If $x\le_Q y$, then $x=x\wedge y$, and by Lemma \ref{qbe-10} and $(BE_4)$ we get: \\
$\hspace*{2.00cm}$ $x\ra y=(x\wedge y)\ra y=((x^*\ra y^*)\ra y^*)^*\ra y$ \\
$\hspace*{3.10cm}$ $=y^*\ra ((x^*\ra y^*)\ra y^*)=(x^*\ra y^*)\ra (y^*\ra y^*)$ \\
$\hspace*{3.10cm}$ $=(x^*\ra y^*)\ra 1=1$. \\
Hence $x\le y$. \\
$(5)$, $(6)$ follow by simple computations. \\
$(7)$: Applying Lemma \ref{qbe-10}, we get: \\
$\hspace*{2.10cm}$ $x\wedge (y\wedge x)=((x^*\ra (y\wedge x)^*)\ra (y\wedge x)^*)^*$ \\
$\hspace*{4.00cm}$ $=((x^*\ra (y^*\vee x^*))\ra (y\wedge x)^*)^*$ \\
$\hspace*{4.00cm}$ $=((x^*\ra ((y^*\ra x^*)\ra x^*))\ra (y\wedge x)^*)^*$ \\
$\hspace*{4.00cm}$ $=((y^*\ra x^*)\ra (x^*\ra x^*)\ra (y\wedge x)^*)^*$ \\
$\hspace*{4.00cm}$ $=(((y^*\ra x^*)\ra 1)\ra (y\wedge x)^*)^*$ \\
$\hspace*{4.00cm}$ $=(1\ra (y\wedge x)^*)^*=y\wedge x$, and \\
$\hspace*{2.10cm}$ $x\wedge (x\wedge y)=((x^*\ra (x\wedge y)^*)\ra (x\wedge y)^*)^*$ \\
$\hspace*{4.00cm}$ $=((x^*\ra ((x^*\ra y^*)\ra y^*))\ra (x\wedge y)^*)^*$ \\
$\hspace*{4.00cm}$ $=(((x^*\ra y^*)\ra (x^*\ra y^*))\ra (x\wedge y)^*)^*$ \\
$\hspace*{4.00cm}$ $=(1\ra (x\wedge y)^*)^*=(x\wedge y)^{**}=x\wedge y$. 
\end{proof}

A BE algebra $X$ is said to be \emph{self distributive} if $x\ra (y\ra z)=(x\ra y)\ra (x\ra z)$, 
for all $x,y,z\in X$ (\cite[Def. 7]{Kim1}).  
Note that the property ``self distributive" of BE algebras is the same as the property ``positive implicative" of 
BCK algebras (\cite{Iseki}). 

\begin{proposition} \label{qbe-30} Let $X$ be an involutive BE algebra. 
Then, the following hold, for all $x,y,z\in X$: \\
$(1)$ $(x\wedge y)\ra z=(y\ra x)\ra (y\ra z);$ \\
$(2)$ $z\ra (x\vee y)=(x\ra y)\ra (z\ra y);$ \\
$(3)$ $x\wedge y\le x,y\le x\vee y$. \\
If $X$ is self distributive, then: \\
$(4)$ $x\wedge y=y\wedge x$ and $x\vee y=y\vee x$. 
\end{proposition}
\begin{proof}
$(1)$: Using Lemma \ref{qbe-10}, we get: \\
$\hspace*{1.00cm}$ $(x\wedge y)\ra z=((x^*\ra y^*)\ra y^*)^*\ra z=z^*\ra ((x^*\ra y^*)\ra y^*)$ \\
$\hspace*{3.00cm}$ $=(x^*\ra y^*)\ra (z^*\ra y^*)=(y\ra x)\ra (y\ra z)$. \\
$(2)$: By $(BE_4)$, we have: \\
$\hspace*{2.00cm}$ $z\ra (x\vee y)=z\ra ((x\ra y)\ra y)=(x\ra y)\ra (z\ra y)$. \\
$(3)$: It follows from $(1)$ and $(2)$ replacing $z$ by $x$ and $y$, respectively. \\
$(4)$: Since $X$ is self distributive, using $(BE_4)$ we have: \\
$\hspace*{2.00cm}$ $(x\ra y)\ra (x\ra z)=(y\ra x)\ra (y\ra z)$, \\ 
and by $(1)$ we get: $(x\wedge y)\ra z=(y\wedge x)\ra z$. 
Taking $z:=0$, it follows that $(x\wedge y)^*=(y\wedge x)^*$, hence $x\wedge y=y\wedge x$. 
Using Proposition \ref{qbe-20}$(3)$, we get $x\vee y=y\vee x$. 
\end{proof}

\begin{remark} \label{qbe-40} 
In general, $x\le y$ does not imply $x\le_Q y$, and the operation $\le_Q$ is not transitive. \\
Indeed consider the following example (see \cite[Cor. 2.1.3]{DvPu}). \\
Let $X=\{0,1,2,3,4,5,6\}$ and let $(X,\ra,^*,1)$ be the involutive BE algebra with $\ra$ and the 
corresponding operation $\wedge$ given in the following tables:  
\[
\begin{array}{c|ccccccc}
\ra & 0 & 1 & 2 & 3 & 4 & 5 & 6 \\ \hline
0   & 1 & 1 & 1 & 1 & 1 & 1 & 1 \\ 
1   & 0 & 1 & 2 & 3 & 4 & 5 & 6 \\ 
2   & 2 & 1 & 1 & 1 & 1 & 1 & 3 \\ 
3   & 4 & 1 & 1 & 1 & 2 & 3 & 5 \\
4   & 3 & 1 & 1 & 1 & 1 & 1 & 1 \\
5   & 6 & 1 & 3 & 1 & 5 & 1 & 1 \\
6   & 5 & 1 & 1 & 1 & 3 & 1 & 1 
\end{array}
\hspace{10mm}
\begin{array}{c|ccccccc}
\wedge & 0 & 1 & 2 & 3 & 4 & 5 & 6 \\ \hline
0    & 0 & 0 & 0 & 0 & 0 & 0 & 0 \\ 
1    & 0 & 1 & 2 & 3 & 4 & 5 & 6 \\ 
2    & 0 & 2 & 2 & 3 & 4 & 6 & 6 \\ 
3    & 0 & 3 & 2 & 0 & 4 & 0 & 6 \\
4    & 0 & 4 & 2 & 0 & 4 & 0 & 4 \\
5    & 0 & 5 & 2 & 2 & 4 & 5 & 6 \\
6    & 0 & 6 & 0 & 6 & 4 & 5 & 6 
\end{array}
\]
where $x^*=x\ra 0$. 
We can see that $3\ra 2=1$, hence $3\le 2$. But $3\wedge 2=2\ne 3$, that is $3\nleq_Q 2$. 
We also have $4=4\wedge 6$ and $6=6\wedge 3$, so $4\le_Q 6$ and $6\le_Q 3$. But $4\neq 0=4\wedge 3$, that is $4\nleq_Q 3$. 
It follows that $\le_Q$ is not transitive. 
\end{remark}

\begin{definition} \label{qbe-70} 
A \emph{(left-)quantum-Wajsberg algebra} (\emph{QW algebra, for short}) is an involutive BE algebra 
$(X,\ra,^*,1)$ verifying the following axiom, for all $x,y,z\in X$: \\
(QW) $x\ra ((x\wedge y)\wedge (z\wedge x))=(x\ra y)\wedge (x\ra z)$. 
\end{definition}

Unless otherwise stated, in what follows, by MV algebras, Wajsberg algebras, quantum-MV algebras and quantum-Wajsberg algebras we understand the left-MV algebras, left-Wajsberg algebras, left-quantum-MV algebras and left-quantum-Wajsberg algebras, respectively.

\begin{proposition} \label{qbe-70-10} In any quantum-Wajsberg algebra $X$, the following holds, for all $x,y\in X$:\\
($QW_1$) $x\ra (x\wedge y)=x\ra y$. 
\end{proposition}
\begin{proof}
Using (QW) and Proposition \ref{qbe-20}$(7)$, we get: \\
$\hspace*{0.6cm}$ $x\ra (x\wedge y)=1\wedge(x\ra (x\wedge y))=(x\ra 1)\wedge (x\ra (x\wedge y))$ \\
$\hspace*{2.5cm}$ $=x\ra ((x\wedge 1)\wedge ((x\wedge y)\wedge x))=x\ra (x\wedge ((x\wedge y)\wedge x))$ \\
$\hspace*{2.5cm}$ $=x\ra ((x\wedge y)\wedge x)=x\ra ((x\wedge y)\wedge (1\wedge x))$ \\
$\hspace*{2.5cm}$ $=(x\ra y)\wedge (x\ra 1)=(x\ra y)\wedge 1=x\ra y$. \\
(Replacing $y$ by $x\wedge y$ in Proposition \ref{qbe-20}$(7)$, we have $x\wedge ((x\wedge y)\wedge x)=(x\wedge y)\wedge x$).  
\end{proof}

\begin{proposition} \label{qbe-70-20} Let $X$ be an involutive BE algebra satisfying property ($QW_1$). 
Then, the following holds, for all $x,y,z\in X$:\\
$\hspace*{3cm}$ $(x\wedge y)\wedge (z\wedge x)=y\wedge (z\wedge x)$. 
\end{proposition}
\begin{proof}
Using ($QW_1$) and $(BE_4)$, we get: \\
$\hspace*{1.30cm}$ $(x\wedge y)\wedge (z\wedge x)=(((x\wedge y)^*\ra (z\wedge x)^*)\ra (z\wedge x)^*)^*$ \\
$\hspace*{4.00cm}$ $=(((x\wedge y)^*\ra (z^*\vee x^*))\ra (z\wedge x)^*)^*$ \\
$\hspace*{4.00cm}$ $=(((x\wedge y)^*\ra ((z^*\ra x^*)\ra x^*))\ra (z\wedge x)^*)^*$ \\
$\hspace*{4.00cm}$ $=(((z^*\ra x^*)\ra ((x\wedge y)^*\ra x^*))\ra (z\wedge x)^*)^*$ \\
$\hspace*{4.00cm}$ $=(((z^*\ra x^*)\ra (x\ra (x\wedge y)))\ra (z\wedge x)^*)^*$ \\
$\hspace*{4.00cm}$ $=(((z^*\ra x^*)\ra (x\ra y))\ra (z\wedge x)^*)^*$ \\
$\hspace*{4.00cm}$ $=(((z^*\ra x^*)\ra (y^*\ra x^*))\ra (z\wedge x)^*)^*$ \\
$\hspace*{4.00cm}$ $=((y^*\ra ((z^*\ra x^*)\ra x^*))\ra (z\wedge x)^*)^*$ \\
$\hspace*{4.00cm}$ $=((y^*\ra (z^*\vee x^*))\ra (z\wedge x)^*)^*$ \\
$\hspace*{4.00cm}$ $=((y^*\ra (z\wedge x)^*)\ra (z\wedge x)^*)^*$ \\
$\hspace*{4.00cm}$ $=y\wedge (z\wedge x)$. \\
(By ($QW_1$), $x\ra (x\wedge y)=x\ra y$, and by $(BE_4)$, 
$(z^*\ra x^*)\ra (y^*\ra x^*)=y^*\ra ((z^*\ra x^*)\ra x^*)=y^*\ra (z^*\vee x^*)$).
\end{proof}

\begin{corollary} \label{qbe-70-30} In any quantum-Wajsberg algebra $X$ the following holds, for all $x,y,z\in X$: \\
($QW_2$) $x\ra (y\wedge (z\wedge x))=(x\ra y)\wedge (x\ra z)$.
\end{corollary}

\begin{theorem} \label{qbe-70-40} Let $X$ be an involutive BE algebra.
Then, condition (QW) is equivalent to conditions ($QW_1$) and ($QW_2$).   
\end{theorem}
\begin{proof}
If $X$ satisfies (QW), then ($QW_1$) and ($QW_2$) follow by Proposition \ref{qbe-70-10} and 
Corollary \ref{qbe-70-30}. 
Conversely, let $X$ be an involutive BE algebra satisfying conditions ($QW_1$) and ($QW_2$).  
By ($QW_1$), using Proposition \ref{qbe-70-20} we get $(x\wedge y)\wedge (z\wedge x)=y\wedge (z\wedge x)$. 
Applying ($QW_2$), we get $x\ra ((x\wedge y)\wedge (z\wedge x))=x\ra (y\wedge (z\wedge x))=(x\ra y)\wedge (x\ra z)$, 
that is (QW). 
\end{proof}

\begin{proposition} \label{qbe-80} Let $X$ be a quantum-Wajsberg algebra. The following hold, for all $x,y\in X$:\\
$(1)$ $x\ra (y\wedge x)=x\ra y;$ \\
$(2)$ $x\le_Q x^*\ra y$ and $x\le_Q y\ra x;$ \\
$(3)$ $x\ra y=0$ iff $x=1$ and $y=0;$ \\
$(4)$ $(x\ra y)^*\wedge x=(x\ra y)^*;$ \\ 
$(5)$ $(x\wedge y)\wedge y=x\wedge y;$ \\
$(6)$ $x\wedge y\le_Q y\le_Q x\vee y;$ \\
$(7)$ $x\vee (y\wedge x)=x;$ \\
$(8)$ $x\wedge (y\vee x)=x;$ \\
$(9)$ $(x\ra y)\ra (y\wedge x)=x;$ \\
$(10)$ $(x\vee y)\ra y=x\ra y;$ \\
$(11)$ $(x\vee y)\vee y=x\vee y;$ \\
$(12)$ $x\le y$ iff $y\wedge x=x$.      
\end{proposition}
\begin{proof}
$(1)$: Applying (QW) and Proposition \ref{qbe-20}$(7)$, we have: \\
$\hspace*{1.00cm}$ $x\ra (y\wedge x)=x\ra (x\wedge (y\wedge x))=x\ra ((x\wedge 1)\wedge (y\wedge x))$ \\
$\hspace*{4.00cm}$ $=(x\ra 1)\wedge (x\ra y)=1\wedge (x\ra y)=x\ra y$. \\
$(2)$: By (QW) we have: \\
$\hspace*{2.00cm}$ $x\wedge (x^*\ra y)=x^{**}\wedge (x^*\ra y)=(x^*\ra 0)\wedge (x^*\ra y)$ \\
$\hspace*{4.15cm}$ $=x^*\ra ((x^*\wedge 0)\wedge (y\wedge x^*))$ \\
$\hspace*{4.15cm}$ $=x^*\ra (0\wedge (y\wedge x^*))=x^*\ra 0=x^{**}=x$. \\
Hence $x\le_Q x^*\ra y$. 
It follows that $x\le_Q y^*\ra x$, and replacing $y$ by $y^*$ we get $x\le_Q y\ra x$. \\
$(3)$: If $x\ra y=0$, then by $(2)$, $x^*\le_Q x\ra y=0$, and using Proposition \ref{qbe-20}$(4)$, we get 
$x^*\le 0$, so $x^*=0$, that is $x=1$. It follows that $1\ra y=0$, hence $y=0$. 
Conversely, $x=1$ and $y=0$ implies $x\ra y=0$. \\
$(4)$: Taking into consideration Lemma \ref{qbe-10} and $(1)$, we get: \\
$\hspace*{2.00cm}$ $(x\ra y)^*\wedge x=(((x\ra y)\ra x^*)\ra x^*)^*$ \\
$\hspace*{4.35cm}$ $=(((y^*\ra x^*)\ra x^*)\ra x^*)^*$ \\
$\hspace*{4.35cm}$ $=(x\ra ((y^*\ra x^*)\ra x^*)^*)^*$ \\
$\hspace*{4.35cm}$ $=(x\ra (y\wedge x))^*$ \\
$\hspace*{4.35cm}$ $=(x\ra y)^*$. \\ 
$(5)$: Using $(4)$ we have: \\
$\hspace*{2.00cm}$ $(x\wedge y)\wedge y=((x^*\ra y^*)\ra y^*)^*\wedge y$ \\
$\hspace*{3.85cm}$ $=(y\ra (x^*\ra y^*)^*)^*\wedge y$ \\
$\hspace*{3.85cm}$ $=(y\ra (x^*\ra y^*)^*)^*$ \\
$\hspace*{3.85cm}$ $=((x^*\ra y^*)\ra y^*)^*$ \\
$\hspace*{3.85cm}$ $=x\wedge y$. \\
$(6)$: By $(5)$, $x\wedge y\le_Q y$. Using $(2)$, we get $y\le_Q (x\ra y)\ra y=x\vee y$. \\
$(7)$: By $(2)$ and $(1)$, we have: \\
$\hspace*{2.00cm}$ $x=x^{**}=(x^*)^*=(x^*\wedge (x\ra y))^*=x\vee (x\ra y)^*$ \\
$\hspace*{2.30cm}$ $=(x\ra (x\ra y)^*)\ra (x\ra y)^*$ \\
$\hspace*{2.30cm}$ $=(x\ra y)\ra (x\ra (x\ra y)^*)^*$ \\
$\hspace*{2.30cm}$ $=(x\ra y)\ra ((x\ra y)\ra x^*)^*$ \\
$\hspace*{2.30cm}$ $=(x\ra y)\ra ((y^*\ra x^*)\ra x^*)^*$ \\
$\hspace*{2.30cm}$ $=(x\ra y)\ra (y\wedge x)$ \\
$\hspace*{2.30cm}$ $=(x\ra (y\wedge x))\ra (y\wedge x)$ \\
$\hspace*{2.30cm}$ $=x\vee (y\wedge x)$. \\
$(8)$: Applying $(7)$, we have: \\
$\hspace*{1.00cm}$ $x\wedge (y\vee x)=(x^*\vee (y\vee x)^*)^*=(x^*\vee (y^*\wedge x^*))^*=(x^*)^*=x$. \\
$(9)$: Using $(1)$ and $(7)$, we get: \\
$\hspace*{1.00cm}$ $(x\ra y)\ra (y\wedge x)=(x\ra (y\wedge x))\ra (y\wedge x)=x\vee (y\wedge x)=x$. \\
$(10)$: Applying $(1)$ we get: \\
$\hspace*{2.00cm}$ $(x\vee y)\ra y=y^*\ra (x\vee y)^*=y^*\ra (x^*\wedge y^*)$ \\
$\hspace*{4.00cm}$ $=y^*\ra x^*=x\ra y$. \\
$(11)$: Using $(10)$ we have: \\
$\hspace*{1.00cm}$  $(x\vee y)\vee y=((x\vee y)\ra y)\ra y=(x\ra y)\ra y=x\vee y$. \\
$(12)$: If $x\le y$, then $x\ra y=1$ and we have 
$y\wedge x=((y^*\ra x^*)\ra x^*)^*=((x\ra y)\ra x^*)^*=(1\ra x^*)^*=(x^*)^*=x$. 
Conversely, if $y\wedge x=x$, using $(1)$ we get 
$x\ra y=x\ra (y\wedge x)=x\ra x=1$, hence $x\le y$. 
\end{proof}

\begin{proposition} \label{qbe-90} Let $X$ be a quantum-Wajsberg algebra. 
If $x,y\in X$ such that $x\le_Q y$, then, the following hold, for any $z\in X$:\\
$(1)$ $y=y\vee x;$ \\
$(2)$ $y^*\le_Q x^*;$ \\
$(3)$ $y\ra z\le_Q x\ra z;$ \\
$(4)$ $z\ra x\le_Q z\ra y;$ \\
$(5)$ $x\wedge z\le_Q y\wedge z;$ \\
$(6)$ $x\vee z\le_Q y\vee z$. 
\end{proposition}
\begin{proof}
$(1)$: From $x\le_Q y$ we have $x=x\wedge y$ and by Proposition \ref{qbe-80}$(7)$, 
      $y=y\vee (x\wedge y)=y\vee x$. \\
$(2)$: Using $(1)$, we get $y=y\vee x$, so that $y^*=(y\vee x)^*=y^*\wedge x^*$. 
Hence $y^*\le_Q x^*$. \\
$(3)$: $x\le_Q y$ implies $x=x\wedge y$. Since by Proposition \ref{qbe-80}$(2)$, 
$y\ra z\le_Q (y\ra z)^*\ra (y\ra x)^*$,  we have: \\
$\hspace*{2.00cm}$ $(y\ra z)\wedge (x\ra z)=(y\ra z)\wedge ((x\wedge y)\ra z)$ \\
$\hspace*{5.15cm}$ $=(y\ra z)\wedge(((x^*\ra y^*)\ra y^*)^*\ra z)$ \\
$\hspace*{5.15cm}$ $=(y\ra z)\wedge (z^*\ra ((x^*\ra y^*)\ra y^*))$ \\
$\hspace*{5.15cm}$ $=(y\ra z)\wedge ((x^*\ra y^*)\ra (z^*\ra y^*))$ \\
$\hspace*{5.15cm}$ $=(y\ra z)\wedge ((y\ra x)\ra (y\ra z))$ \\
$\hspace*{5.15cm}$ $=(y\ra z)\wedge ((y\ra z)^*\ra (y\ra x)^*)$ \\
$\hspace*{5.15cm}$ $=y\ra z$. \\
It follows that $y\ra z\le_Q x\ra z$. \\
$(4)$: By $(1)$, $x\le_Q y$ implies $y=y\vee x$. Since by Proposition \ref{qbe-80}$(2)$, 
$z\ra x\le_Q (z\ra x)^*\ra (y\ra x)^*$, we have: \\
$\hspace*{2.00cm}$ $(z\ra x)\wedge (z\ra y)=(z\ra x)\wedge (z\ra (y\vee x))$ \\
$\hspace*{5.15cm}$ $=(z\ra x)\wedge (z\ra ((y\ra x)\ra x))$ \\
$\hspace*{5.15cm}$ $=(z\ra x)\wedge ((y\ra x)\ra (z\ra x))$ \\
$\hspace*{5.15cm}$ $=(z\ra x)\wedge ((z\ra x)^*\ra (y\ra x)^*)$ \\
$\hspace*{5.15cm}$ $=z\ra x$. \\
Thus $z\ra x\le_Q z\ra y$. \\
$(5)$: From $x\le_Q y$ we have $y^*\le_Q x^*$. By $(3)$ we get $x^*\ra z^*\le_Q y^*\ra z^*$ and 
$(y^*\ra z^*)\ra z^*\le_Q (x^*\ra z^*)\ra z^*$. 
It follows that $((x^*\ra z^*)\ra z^*)^*\le_Q ((y^*\ra z^*)\ra z^*)^*$, 
that is $x\wedge z\le_Q y\wedge z$. \\
$(6)$: By $(3)$, we have $y\ra z\le_Q x\ra z$ and $(x\ra z)\ra z\le_Q (y\ra z)\ra z$. 
Hence $x\vee z\le_Q y\vee z$.  
\end{proof}

\begin{proposition} \label{qbe-100} Let $X$ be a quantum-Wajsberg algebra. Then, the following hold, for all 
$x,y,z\in X$:\\
$(1)$ $(x\wedge y)\wedge (y\wedge z)=(x\wedge y)\wedge z;$ \\
$(2)$ $\le_Q$ is transitive; \\
$(3)$ $x\vee y\le_Q x^*\ra y;$ \\
$(4)$ $(x^*\ra y)^*\ra (x\ra y^*)^*=x^*\ra y;$ \\
$(5)$ $(x\ra y)^*\ra (y\ra x)^*=x\ra y;$ \\
$(6)$ $(y\ra x)\ra (x\ra y)=x\ra y;$ \\
$(7)$ $(x\ra y)\vee (y\ra x)=1;$ \\
$(8)$ $(z\wedge x)\ra (y\wedge x)=(z\wedge x)\ra y$.    
\end{proposition}
\begin{proof}
$(1)$: Applying Propositions \ref{qbe-30}$(1)$ and \ref{qbe-80}$(6)$, we get: \\
$\hspace*{0.35cm}$ $(x\wedge y)\wedge (y\wedge z)=(((x\wedge y)^*\ra (y\wedge z)^*)\ra (y\wedge z)^*)^*$ \\
$\hspace*{3.00cm}$ $=(((x\wedge y)^*\ra (y^*\vee z^*))\ra (y^*\vee z^*))^*$ \\
$\hspace*{3.00cm}$ $=(((x\wedge y)^*\ra ((y^*\ra z^*)\ra z^*))\ra ((y^*\ra z^*)\ra z^*))^*$ \\
$\hspace*{3.00cm}$ $=(((y^*\ra z^*)\ra ((x\wedge y)^*\ra z^*))\ra ((y^*\ra z^*)\ra z^*))^*$ \\
$\hspace*{3.00cm}$ $=((((x\wedge y)^*\ra z^*)\wedge (y^*\ra z^*))\ra z^*)^*$ \\
$\hspace*{3.00cm}$ $=(((x\wedge y)^*\ra z^*)\ra z^*)^*$ \\
$\hspace*{3.00cm}$ $=(x\wedge y)\wedge z$. \\
(Denoting by $a=(x\wedge y)\ra z^*$, $b=y^*\ra z^*$, $c=z^*$, by Proposition \ref{qbe-30}$(1)$ we have 
$(b\ra a)\ra (b\ra c)=(a\wedge b)\ra c$. It follows that: \\
$((y^*\ra z^*)\ra ((x\wedge y)^*\ra z^*))\ra ((y^*\ra z^*)\ra z^*)=(((x\wedge y)^*\ra z^*)\wedge (y^*\ra z^*))\ra z^*$. \\
By Proposition \ref{qbe-80}$(6)$, $x\wedge y\le_Q y$, so that $y^*\le_Q (x\wedge y)^*$. 
Hence, by Proposition \ref{qbe-90}$(3)$, $(x\wedge y)^*\ra z^*\le_Q y^*\ra z^*$, that is 
$((x\wedge y)^*\ra z^*)\wedge (y^*\ra z^*)=(x\wedge y)^*\ra z^*$). \\
$(2)$: If $x\le_Q y$ and $y\le_Q z$, then $x=x\wedge y$ and $y=y\wedge z$. 
Applying $(1)$, we get $x=(x\wedge y)\wedge (y\wedge z)=(x\wedge y)\wedge z=x\wedge z$, that is $x\le_Q z$. 
We conclude that $\le_Q$ is transitive. \\
$(3)$: Since $x^*\le_Q x\ra y$, we have $(x\ra y)\ra y\le_Q x^*\ra y$, that is $x\vee y\le_Q x^*\ra y$. \\
$(4)$: By $(QW_1)$ we have $y^*\ra x=y^*\ra (y^*\wedge x)$. 
Using Lemma \ref{qbe-10}$(5)$,$(6)$ and Propositions \ref{qbe-80}$(2)$, \ref{qbe-90}$(1)$, we get: \\
$\hspace*{0.70cm}$ $(x^*\ra y)^*\ra (x\ra y^*)^*=(y^*\ra x)^*\ra (x\ra y^*)^*$ \\
$\hspace*{4.50cm}$ $=(y^*\ra (y^*\wedge x))^*\ra (x\ra y^*)^*$ \\
$\hspace*{4.50cm}$ $=(x\ra y^*)\ra (y^*\ra (y^*\wedge x))$ \\
$\hspace*{4.50cm}$ $=y^*\ra ((x\ra y^*)\ra (y^*\wedge x))$ \\
$\hspace*{4.50cm}$ $=y^*\ra ((y^*\wedge x)^*\ra (x\ra y^*)^*)$ \\
$\hspace*{4.50cm}$ $=y^*\ra ((y\vee x^*)\ra (x\ra y^*)^*)$ \\
$\hspace*{4.50cm}$ $=y^*\ra (((y\ra x^*)\ra x^*)\ra (x\ra y^*)^*)$ \\
$\hspace*{4.50cm}$ $=y^*\ra (((x\ra y^*)\ra x^*)\ra (x\ra y^*)^*)$ \\
$\hspace*{4.50cm}$ $=y^*\ra ((x\ra (x\ra y^*)^*)\ra (x\ra y^*)^*)$ \\
$\hspace*{4.50cm}$ $=y^*\ra (x\vee (x\ra y^*)^*)$ \\
$\hspace*{4.50cm}$ $=y^*\ra x=x^*\ra y$. \\
(By Proposition \ref{qbe-80}$(2)$, $x^*\le_Q x\ra y^*$, so $(x\ra y^*)^*\le_Q x$. 
Using Proposition \ref{qbe-90}$(1)$, we get $x=x\vee (x\ra y^*)^*$). \\
$(5)$: It follows from $(4)$ replacing $x$ by $x^*$ and using Lemma \ref{qbe-10}$(5)$,$(6)$. \\
$(6)$: It follows from $(5)$ and Lemma \ref{qbe-10}$(6)$. \\ 
$(7)$: Applying $(6)$, we have: \\
$\hspace*{2.00cm}$ $(x\ra y)\vee (y\ra x)=((x\ra y)\ra (y\ra x))\ra (y\ra x)$ \\
$\hspace*{5.15cm}$ $=(y\ra x)\ra (y\ra x)=1$. \\
$(8)$: Applying Lemma \ref{qbe-10}, $(BE_4)$ and Proposition \ref{qbe-80}$(1)$, we get: \\
$\hspace*{2.00cm}$ $(z\wedge x)\ra (y\wedge x)=((z^*\ra x^*)\ra x^*)^*\ra (y\wedge x)$ \\
$\hspace*{5.00cm}$ $=(y\wedge x)^*\ra ((z^*\ra x^*)\ra x^*)$ \\
$\hspace*{5.00cm}$ $=(z^*\ra x^*)\ra ((y\wedge x)^*\ra x^*)$ \\
$\hspace*{5.00cm}$ $=(z^*\ra x^*)\ra (x\ra (y\wedge x))$ \\ 
$\hspace*{5.00cm}$ $=(z^*\ra x^*)\ra (x\ra y)$ \\
$\hspace*{5.00cm}$ $=(z^*\ra x^*)\ra (y^*\ra x^*)$ \\
$\hspace*{5.00cm}$ $=y^*\ra ((z^*\ra x^*)\ra x^*)$ \\
$\hspace*{5.00cm}$ $=((z^*\ra x^*)\ra x^*)^*\ra y$ \\
$\hspace*{5.00cm}$ $=(z\wedge x)\ra y$. 
\end{proof}

\begin{corollary} \label{qbe-100-10} If $X$ is a quantum-Wajsberg algebra, then $\le_Q$ is a partial order on $X$. 
\end{corollary}
\begin{proof}
It follows by Propositions \ref{qbe-20}$(2)$ and \ref{qbe-100}$(2)$. 
\end{proof}

\begin{example} \label{qbe-110}
The involutive BE algebra from Example \ref{qbe-40} is not a quantum-Wajsberg algebra. 
Indeed, $2\ra ((2\wedge 0)\wedge (6\wedge 0))=2\ne 3=(2\ra 0)\wedge (2\ra 6)$, thus axiom (QW) is not satisfied. 
\end{example}

\begin{remark} \label{qbe-110-10}
As we can see in the next example, in general the quantum-Wajsberg algebras are not (commutative) quantum-B algebras. 
We will prove in the next section that the $\vee$-commutative quantum-Wajsberg algebras are commutative quantum-B algebras. 
\end{remark}

\begin{example} \label{qbe-120} 
Let $X=\{0,a,b,c,1\}$ and let $(X,\ra,^*,1)$ be the involutive BE algebra with $\ra$ and the corresponding 
operation $\wedge$ given in the following tables:  
\[
\begin{array}{c|ccccccc}
\ra & 0 & a & b & c & 1 \\ \hline
0   & 1 & 1 & 1 & 1 & 1 \\ 
a   & b & 1 & a & 1 & 1 \\ 
b   & a & 1 & 1 & 1 & 1 \\ 
c   & c & 1 & 1 & 1 & 1 \\
1   & 0 & a & b & c & 1  
\end{array}
\hspace{10mm}
\begin{array}{c|ccccccc}
\wedge & 0 & a & b & c & 1 \\ \hline
0    & 0 & 0 & 0 & 0 & 0 \\ 
a    & 0 & a & b & c & a \\ 
b    & 0 & b & b & c & b \\ 
c    & 0 & a & b & c & c \\
1    & 0 & a & b & c & 1 
\end{array}
\]
where $x^*=x\ra 0$. Then, $X$ is a quantum-Wajsberg algebra.  
We can see that $c\le b$, but $a\ra c=1\nleq a=a\ra b$, hence axiom $(QB_2)$ is not satisfied.  
It follows that $X$ is not a commutative quantum-B algebra. 
\end{example}

$\vspace*{1mm}$

\section{Connections with Wajsberg algebras}

Since the $\vee$-commutativity plays an important role in the study of quantum-Wajsberg algebras, we define the $\vee$-commutative quantum-Wajsberg algebras and we investigate their properties. 
We give a condition for a quantum-Wajsberg algebra to be a Wajsberg algebra, and we prove  
that a Wajsberg algebra is both a quantum-Wajsberg algebra and a commutative quantum-B algebra. \\
The $\vee$-commutative BE algebras have been defined and studied in \cite{Walend1}.

\begin{definition} \label{cqbe-10} 
A BE algebra $X$ is called \emph{$\vee$-commutative} if it satisfies axiom $($$\vee$-$comm)$. 
\end{definition}

\begin{remark} \label{qbe-130}
It is known that: \\
$\hspace*{0.5cm}$ $-$ $\vee$-commutative BE algebras are $\vee$-commutative BCK algebras (\cite{Walend1}]),  \\
$\hspace*{0.5cm}$ $-$ bounded $\vee$-commutative BCK are term-equivalent to MV algebras (\cite{Mund1}) and \\
$\hspace*{0.5cm}$ $-$ Wajsberg algebras are term-equivalent to MV algebras (\cite{Font1}). \\
It follows that the bounded $\vee$-commutative BE algebras are bounded $\vee$-commutative BCK algebras, hence are term-equivalent with MV algebras, hence with Wajsberg algebras. 
Hence the $\vee$-commutative quantum-Wajsberg algebras are the Wajsberg algebras. 
\end{remark} 

By Proposition \ref{qbe-30}, a self distributive involutive BE algebra is $\vee$-commutative, so it is a 
Wajsberg algebra. 

\begin{proposition} \label{cqbe-20} 
If $(X,\ra,^*,1)$ is a Wajsberg algebra, then the following hold, for all $x,y\in X$: \\
$(1)$ $X$ is involutive; \\
$(2)$ $x\wedge y=y\wedge x;$  \\
$(3)$ $x\le y$ implies $x\le_Q y;$ \\
$(4)$ $\le$ is transitive. 
\end{proposition}
\begin{proof} 
$(1)$: See \cite[Th. 3.3.11]{Ior35}. \\
$(2)$: It follows by Definition \ref{pr-70}. \\
$(3)$: If $x\le y$, then $x\ra y=1$ and we have: \\
$\hspace*{0.50cm}$ $x\wedge y=((x^*\ra y^*)\ra y^*)^*=((y^*\ra x^*)\ra x^*)^*=(x\ra (y^*\ra x^*)^*)^*$ \\
$\hspace*{1.50cm}$ $=(x\ra (x\ra y)^*)^*=(x\ra 1^*)^*=x^{**}=x$. \\
Hence $x\le_Q y$. \\
$(4)$: Let $x\le y$ and $y\le z$, that is $x\ra y=1$ and $y\ra z=1$. 
Then we have: \\
$\hspace*{0.50cm}$ $x\ra z=x\ra (1\ra z)=x\ra ((y\ra z)\ra z)=x\ra ((z\ra y)\ra y)$ \\
$\hspace*{1.60cm}$ $=(z\ra y)\ra (x\ra y)=(z\ra y)\ra1=1$. \\
It follows that $x\le z$, that is $\le$ is transitive. 
\end{proof}

\begin{corollary} \label{cqbe-30} If $(X,\ra,^*,1)$ is a Wajsberg algebra, then $\le_Q\Longleftrightarrow \le$.
\end{corollary}
\begin{proof} 
By $(W_5)$, $(W_6)$, $(W_1)$, $(W_7)$ and $(W_3)$, it follows that $X$ is a $\vee$-commutative BE algebra, so that $X$ 
is an involutive BE algebra. 
Using Proposition \ref{qbe-20}$(4)$, $x\le_Q y$ implies $x\le y$. According to Proposition \ref{cqbe-20}$(3)$, 
$x\le y$ implies $x\le_Q y$. 
Hence $\le_Q\Longleftrightarrow \le$. 
\end{proof}

\begin{proposition} \label{cqbe-60} 
Let $(X,\ra,^*,1)$ be a Wajsberg algebra. Then, the following hold, for all $x,y,z\in X$: \\
$(1)$ $(y\vee x)\ra y=x\ra y;$ \\
$(2)$ $x\wedge (x\vee y)=x;$ \\
$(3)$ $(y\vee x)\vee y=x\vee y;$ \\
$(4)$ $(x\wedge y)\wedge z=y\wedge (x\wedge z);$ \\
$(5)$ $x\ra ((x\wedge y)\wedge x)=x\ra y;$ \\ 
$(6)$ $x\wedge (x^*\ra y)=x;$ \\            
$(7)$ $(x\ra y)\ra (x\wedge y)=x;$ \\       
$(8)$ $(z\wedge x)\ra (y\wedge x)=(z\wedge x)\ra y;$ \\  
$(9)$ $(x\ra y)^*\wedge x=(x\ra y)^*;$ \\   
$(10)$ $(x\wedge y)\wedge y=x\wedge y;$ \\       
$(11)$ $x\wedge y\le_Q y \le_Q x\vee y;$ \\  
$(12)$ $(x\wedge y)\wedge (y\wedge z)=(x\wedge y)\wedge z$.   
\end{proposition}
\begin{proof}
$(1)$, $(2)$, $(3)$ follow applying the $\vee$-commutativity in Proposition \ref{qbe-80}$(10)$, $(8)$, $(11)$, respectively. \\
$(4)$: By $\vee$-commutativity and Proposition \ref{qbe-100}$(1)$, $(8)$, we have: \\
$\hspace*{2.10cm}$ $(x\wedge y)\wedge z=(y\wedge x)\wedge z=(y\wedge x)\wedge (x\wedge z)$ \\
$\hspace*{4.00cm}$ $=(((y\wedge x)^*\ra (x\wedge z)^*)\ra (x\wedge z)^*)^*$ \\
$\hspace*{4.00cm}$ $=((z\wedge x)\ra (y\wedge x))\ra (z\wedge x)^*)^*$ \\
$\hspace*{4.00cm}$ $=(((z\wedge x)\ra y)\ra (z\wedge x)^*)^*$ \\
$\hspace*{4.00cm}$ $=((y^*\ra (z\wedge x)^*)\ra (z\wedge x)^*)^*$ \\
$\hspace*{4.00cm}$ $=y\wedge (z\wedge x)=y\wedge (x\wedge z)$. \\
$(5)$: Applying Proposition \ref{pr-60}$(4)$, we have: \\
$\hspace*{2.00cm}$ $x\ra (x\wedge y)=(x\wedge y)^*\ra x^*=((x^*\ra y^*)\ra y^*)\ra x^*$ \\
$\hspace*{4.10cm}$ $=((y^*\ra x^*)\ra x^*)\ra x^*=y^*\ra x^*=x\ra y$. \\
By $\vee$-commutativity, we also have $x\ra (y\wedge x)=x\ra y$, and replacing $y$ by $x\wedge y$ we get the second identity. \\
$(6)$: From $0\le y$ we get $x^*\ra 0\le x^*\ra y$, hence $x\le x^*\ra y$. 
By Corollary \ref{cqbe-30}, $x\le_Q x^*\ra y$, so that $x\wedge (x^*\ra y)=x$. \\
$(7)$: Replacing $x$ by $x^*$ in $(6)$, we have $x^*\wedge (x\ra y)=x^*$, and we get: \\
$\hspace*{1.00cm}$ $(x\ra y)\ra (x\wedge y)=(x\wedge y)^*\ra (x\ra y)^*$ \\
$\hspace*{4.00cm}$ $=((x^*\ra y^*)\ra y^*)\ra (x\ra y)^*$ \\
$\hspace*{4.00cm}$ $=((y^*\ra x^*)\ra x^*)\ra (x\ra y)^*$ \\
$\hspace*{4.00cm}$ $=((x\ra y)\ra x^*)\ra (x\ra y)^*$ \\
$\hspace*{4.00cm}$ $=(x\ra (x\ra y)^*)\ra (x\ra y)^*$ \\
$\hspace*{4.00cm}$ $=x\vee (x\ra y)^*=(x^*\wedge (x\ra y))^*=(x^*)^*=x$. \\
$(8)$, $(9)$ follow similar to the proof of Propositions \ref{qbe-100}$(8)$ and \ref{qbe-80}$(4)$, respectively, 
using Lemma \ref{qbe-10}, $(BE_4)$ and $(5)$. \\
$(10)$: Similar to the proof of Proposition \ref{qbe-80}$(5)$, using $(9)$. \\
$(11)$: By $(10)$, $x\wedge y\le_Q y$. Applying $(6)$, $y\le_Q y^*\ra (x\ra y)^*=(x\ra y)\ra y=x\vee y$. \\
$(12)$: Similar to the proof of Proposition \ref{qbe-100}$(1)$, using Proposition \ref{qbe-30}$(1)$ and $(11)$. 
\end{proof}

\begin{theorem} \label{cqbe-70} Let $(X,\ra,^*,1)$ be a quantum-Wajsberg algebra. The following are equivalent: \\
$(a)$ $X$ is $\vee$-commutative; \\
$(b)$ $x\ra y=(y\vee x)\ra y$ and $x\le_Q x\vee y$, for all $x,y\in X$.
\end{theorem}
\begin{proof}
$(a)\Rightarrow (b)$ It follows by Proposition \ref{cqbe-60}$(1)$, $(2)$. \\
$(b)\Rightarrow (a)$
By $(b)$, we have $x\ra y=(y\vee x)\ra y$, hence $(x\ra y)\ra y=((y\vee x)\ra y)\ra y$. 
It follows that $x\vee y=(y\vee x)\vee y$, and similarly $y\vee x=(x\vee y)\vee x$. 
Replacing $x$ by $x\vee y$ and $y$ by $x$ in $x\le_Q x\vee y$ from $(b)$, we get 
$x\vee y\le_Q (x\vee y)\vee x=y\vee x$, and similarly $y\vee x\le_Q (y\vee x)\vee y=x\vee y$. 
Since $\le_Q$ is antisymmetric, we get $x\vee y=y\vee x$, that is $X$ is $\vee$-commutative. 
\end{proof}

\begin{proposition} \label{cqbe-90} Let $(X,\ra,^*,1)$ be a Wajsberg algebra. 
The following hold, for all $x,y,z\in X$: \\
$(1)$ $x\le_Q y$ implies $z\ra x\le_Q z\ra y$, $y\ra z\le_Q x\ra z$ and $y^*\le_Q x^*;$ \\
$(2)$ $x\le_Q y$ and $x\le_Q z$ imply $x\le_Q y\wedge z;$ \\
$(3)$ $y\le_Q x$ and $z\le_Q x$ imply $y\vee z\le_Q x;$ \\
$(4)$ $x\le_Q y$ implies $x\vee z\le_Q y\vee z$ and $x\wedge z\le_Q y\wedge z$. 
\end{proposition}
\begin{proof}
$(1)$: It is clear, since by Corollary \ref{cqbe-30}, $\le_Q \Longleftrightarrow \le$, and $\le$ verifies these 
inequalities in any BCK algebra. \\
$(2)$: Since $x\le_Q y$ and $x\le_Q z$, we have $x=x\wedge y$ and $x=x\wedge z$. 
Using Proposition \ref{cqbe-60}$(12)$ we get $x\wedge (y\wedge z)=(x\wedge y)\wedge (y\wedge z)=(x\wedge y)\wedge z=x\wedge z=x$, 
hence $x\le_Q y\wedge z$. \\
$(3)$: From hypothesis, we get $x^*\le_Q y^*$ and $x^*\le_Q z^*$. 
Using $(2)$ we get $x^*\le_Q y^*\wedge z^*=(y\vee z)^*$, hence $y\vee z\le_Q x$. \\
$(4)$: Applying $(1)$, from $x\le_Q y$ we get $y\ra z\le_Q x\ra z$ and $(x\ra z)\ra z\le_Q (y\ra z)\ra z$, 
that is $x\vee z\le_Q y\vee z$. 
Moreover $x\le_Q y$ implies $y^*\le_Q x^*$, so that $x^*\ra z^*\le_Q y^*\ra z^*$ and 
$(y^*\ra z^*)\ra z^*\le_Q (x^*\ra z^*)\ra z^*$. 
It follows that $((x^*\ra z^*)\ra z^*)^*\le_Q ((y^*\ra z^*)\ra z^*)^*$, hence $x\wedge z\le_Q y\wedge z$. 
\end{proof}

\begin{proposition} \label{cqbe-100} Let $(X,\ra,^*,1)$ be a Wajsberg algebra. 
The following holds, for all $x,y,z\in X$: \\
$\hspace*{3cm}$ $x\ra (y\wedge z)=(x\ra y)\wedge (x\ra z)$. 
\end{proposition}
\begin{proof} 
Let $(X,\ra,^*,1)$ be a Wajsberg algebra, and let $x,y,z\in X$. 
Since $X$ is a bounded $\vee$-commutative BE algebra, by Corollary \ref{cqbe-30} we have 
$\le_Q \Longleftrightarrow \le$. 
Using Propositions \ref{cqbe-90} and \ref{pr-60}$(5)$, for any $u\in X$ we have: \\
$\hspace*{0.20cm}$ $u\le_Q x\ra (y\wedge z)$ iff $u\le_Q (y\wedge z)^*\ra x^*$ iff $u\le_Q (y^*\vee z^*)\ra x^*$ \\
$\hspace*{3.00cm}$ iff $y^*\vee z^*\le_Q u\ra x^*$ \\
$\hspace*{3.00cm}$ iff $y^*\le_Q u\ra x^*$ and $z^*\le_Q u\ra x^*$ \\ 
$\hspace*{3.00cm}$ iff $u\le_Q y^*\ra x^*=x\ra y$ and $u\le_Q z^*\ra x^*=x\ra z$ \\
$\hspace*{3.00cm}$ iff $u\le_Q (x\ra y)\wedge (x\ra z)$. \\
Since in Wajsberg algebras the relation $\le$ is a partial order, we conclude that 
$x\ra (y\wedge z)=(x\ra y)\wedge (x\ra z)$. 
\end{proof}

\begin{proposition} \label{cqbe-110} Any Wajsberg algebra $(X,\ra,^*,1)$ is a quantum-Wajsberg algebra. 
\end{proposition}
\begin{proof} 
Let $(X,\ra,^*,1)$ be a Wajsberg algebra, and let $x,y,z\in X$. 
Using Proposition \ref{cqbe-60}$(12)$,$(4)$,$(5)$ and Proposition \ref{cqbe-100}, we get: \\
$\hspace*{0.90cm}$ $x\ra ((x\wedge y)\wedge (z\wedge x))=x\ra ((y\wedge x)\wedge (x\wedge z))
                        =x\ra ((y\wedge x)\wedge z)$ \\
$\hspace*{4.50cm}$ $=x\ra (x\wedge (y\wedge z))=x\ra (y\wedge z)$ \\
$\hspace*{4.50cm}$ $=(x\ra y)\wedge (x\ra z)$. \\
It follows that axiom $(QW)$ is satisfied, hence $X$ is a quantum-Wajsberg algebra. 
\end{proof}

\begin{theorem} \label{cqbe-120} 
A quantum-Wajsberg algebra $(X,\ra,^*,1)$ is a Wajsberg algebra if and only if, for all $x,y\in X$, 
$x\le y$ implies $x\le_Q y$. 
\end{theorem}
\begin{proof} 
Let $(X,\ra,^*,1)$ be a Wajsberg algebra, and let $x,y\in X$. 
According to Proposition \ref{cqbe-110}, $X$ is a quantum-Wajsberg algebra. 
If $x\le y$, by Proposition \ref{qbe-80}$(12)$ and by $\vee$-commutativity, we have $x=y\wedge x=x\wedge y$, hence $x\le_Q y$. 
Conversely, suppose that $X$ is a quantum-Wajsberg algebra such that, for all $x,y\in X$, $x\le y$ implies $x\le_Q y$. 
Using $(BE_4)$ and $(QW_1)$, we get: \\
$\hspace*{0.65cm}$ $(x\wedge y)\ra (y\wedge x)=(y\wedge x)^*\ra (x\wedge y)^*$ \\
$\hspace*{3.50cm}$ $=(y\wedge x)^*\ra ((x^*\ra y^*)\ra y^*)$ \\
$\hspace*{3.50cm}$ $=(x^*\ra y^*)\ra ((y\wedge x)^*\ra y^*)$ \\
$\hspace*{3.50cm}$ $=(y\ra x)\ra (y\ra (y\wedge x))$ \\
$\hspace*{3.50cm}$ $=(y\ra x)\ra (y\ra x)=1$. \\
It follows that $x\wedge y\le y\wedge x$, hence $x\wedge y\le_Q y\wedge x$, and similarly $y\wedge x\le_Q x\wedge y$. 
Since $\le_Q$ is antisymmetric, we get $x\wedge y=y\wedge x$, that is $X$ is a Wajsberg algebra.    
\end{proof}

\begin{proposition} \label{cqbe-40} 
If $(X,\ra,^*,1)$ is a Wajsberg algebra, then $(X,\le,\ra,1)$ is a unital commutative quantum-B algebra.  
\end{proposition}
\begin{proof} 
Since $X$ is a $\vee$-commutative BE algebra, according to \cite[Cor. 3.5]{Walend1}, it is a ($\vee$-commutative) 
BCK algebra. 
By \cite[Cor. of Prop. 12]{Rump2}, any BCK algebra is a unital (commutative) quantum-B algebra 
(see also \cite[Ex. 3.2]{Rump1}). 
\end{proof}

A quantum-Wajsberg algebra $(X,\ra,^*,1)$ is \emph{self distributive} if $X$ is self distributive as BE algebra. 

\begin{corollary} \label{cqbe-40-10} 
If $(X,\ra,^*,1)$ is a self distributive quantum-Wajsberg algebra, then $(X,\le,\ra,1)$ is a unital commutative quantum-B algebra.  
\end{corollary}

\begin{example} \label{cqbe-50} 
Let $X=\{0,a,b,c,1\}$ and consider the operations $\ra$, $\wedge$ and $^*$ given in the following tables:  
\[
\begin{array}{c|ccccccc}
\ra & 0 & a & b & c & 1 \\ \hline
0   & 1 & 1 & 1 & 1 & 1 \\ 
a   & b & 1 & c & a & 1 \\ 
b   & a & 1 & 1 & 1 & 1 \\ 
c   & c & 1 & a & 1 & 1 \\
1   & 0 & a & b & c & 1  
\end{array}
\hspace{10mm}
\begin{array}{c|ccccccc}
\wedge & 0 & a & b & c & 1 \\ \hline
0    & 0 & 0 & 0 & 0 & 0 \\ 
a    & 0 & a & b & c & a \\ 
b    & 0 & b & b & b & b \\ 
c    & 0 & c & b & c & c \\
1    & 0 & a & b & c & 1 
\end{array}
\hspace{10mm}
\begin{array}{c|c}
x & x^* \\ \hline
0 & 1 \\ 
a & b \\  
b & a \\ 
c & c \\
1 & 0   
\end{array}
\]
Then, $(X,\ra,^*,1)$ is a Wajsberg algebra, and $(X,\le,\ra,1)$ is a commutative quantum-B algebra.
\end{example}

$\vspace*{1mm}$

\section{Connection with quantum-MV algebras}

In this section, we establish the connection between quantum-Wajsberg algebras and quantum-MV algebras, by proving 
that the two quantum structures are term-equivalent.  
We show that, in general, the quantum-MV algebras are not (commutative) quantum-B algebras and we prove that in the case of a self distributive quantum-Wajsberg algebra the corresponding quantum-MV algebra is an MV algebra. 
We deduce certain properties of quantum-MV algebras, and we give conditions for a quantum-MV algebra to be an 
MV algebra. 

\begin{theorem} \label{qmv-30}
The class of (left-)quantum-MV algebras is term-equivalent to the class of (left-)quantum-Wajsberg algebras.  
\end{theorem}
\begin{proof} 
We use the transformations $\Phi$ and $\Psi$ from Remark \ref{qbe-05}. 
Since by Remark \ref{qbe-05} the involutive m-BE algebras are term-equivalent to the involutive BE algebras,  it remains to prove that $(Pqmv)\iff (QW)$. 
Applying $\Phi$, from (Pqmv) we get: \\
$\hspace*{0.50cm}$ $x\odot ((x^*\vee y)\vee (z\vee x^*))=(x\ra ((x^*\vee y)\vee (z\vee x^*))^*)^*$ \\
$\hspace*{4.30cm}$ $=(x\ra ((x\wedge y^*)\wedge (z^*\wedge x)))^*$ and \\
$\hspace*{0.50cm}$ $(x\odot y)\vee (x\odot z)=(x\ra y^*)^*\vee (x\ra z^*)^*=((x\ra y^*)\wedge (x\ra z^*))^*$, \\
hence (Pqmv) becomes: \\
$\hspace*{0.5cm}$ $(x\ra ((x\wedge y^*)\wedge (z^*\wedge x)))^*=((x\ra y^*)\wedge (x\ra z^*))^*$, \\
for all $x,y,z\in X$. Replacing $y$ by $y^*$ and $z$ by $z^*$, we get axiom (QW). 
Similarly, applying $\Psi$, axiom (QW) implies axiom (Pqmv).
\end{proof}

\begin{corollary} \label{qmv-30-10} In general, the quantum-MV algebras are not (commutative) quantum-B algebras. 
\end{corollary}
\begin{proof}
It follows by Theorem \ref{qmv-30} and Remark \ref{qbe-110-10}. 
\end{proof}

\begin{remarks} \label{qmv-40} 
$(1)$ According to Proposition \ref{cqbe-110}, the Wajsberg algebras are quantum-Wajsberg algebras, and 
by Theorem \ref{qmv-30} the quantum-Wajsberg algebras are term-equivalent to quantum-MV algebras. 
Since by Remark \ref{qbe-130}, Wajsberg algebras are term-equivalent to MV algebras, 
it follows that any MV algebra is a quantum-MV algebra (for a directly proof see \cite[Ex. 2.3.14]{DvPu}). \\
$(2)$ Every orthomodular lattice $(L,\vee,\wedge,^{'},0,1)$ determines a quantum-MV algebra by taking $\oplus$ as the 
supremum $\vee$ and $^*$ as the orthocomplement $^{'}$ (\cite[Th. 2.3.9]{DvPu}). 
\end{remarks}

\begin{remarks} \label{qmv-40-10} Let $(X,\odot,^*,1)$ be a (left)-quantum MV algebra, and consider the notations 
(see \cite{Ior35}): \\
$(Pmv)$ $x\odot (x^*\vee_m^M y)=x\odot y$, \\
$(Pq)$ $\hspace*{0.20cm}$ $x\odot (y\vee_m^M (z\vee_m^M x^*))=(x\odot y)\vee_m^M (x\odot z)$, \\
$(Pom)$ $(x\odot y)\oplus ((x\odot y)^*\odot x)=x$ or, equivalently, $x\vee_m^M (x\odot y)=x$, \\
$(OM)$ $\hspace*{0.03cm}$ $x\wedge_m^M (y\ra x)=x$, \\ 
$($$m$-$An)$ $(x\odot y^*=0$ and $y\odot x^*=0)\Longrightarrow x=y$ (antisymmetry), \\
$($$\wedge_m^M$-$comm)$ $x\wedge_m^M y=y\wedge_m^M x$. \\  
$(1)$ According to \cite[Th. 11.2.10]{Ior35}, $(Pqmv)\iff (Pmv)+(Pq)$. 
Moreover, since $(Pq)\iff (Pom)$, then, by \cite[Th. 11.2.18]{Ior35}, $(Pqmv)\iff (Pmv)+(Pom)$. 
Using the transformations $\Phi$ and $\Psi$, we can easily prove that $(Pmv)\iff (QW_1)$, $(Pq)\iff (QW_2)$ and $(Pom)\iff (OM)$. \\
$(2)$ By \cite[Th. 11.3.5]{Ior35}, $(Pmv)$+$($$m$-$An)$ $\Longrightarrow$ $($$\wedge_m^M$-$comm)$ and, consequently, 
by \cite[Th. 11.3.9]{Ior35}, we have $(Pqmv)$+$($$m$-$An)$ $\iff$ $($$\wedge_m^M$-$comm)$ 
(i.e. a quantum-MV algebra is an MV algebra if and only if $($$m$-$An)$ (antisymmetry of $\le_m$) holds). 
\end{remarks}

In what follows, the result from Remarks \ref{qmv-40-10}$(2)$ will be extended to the case of quantum-Wajsberg algebras. 

\begin{proposition} \label{qmv-40-20} Let $(X,\ra,^*,1)$ be an involutive BE algebra satisfying conditions $(QW_1)$ 
and \\ 
$(An)$ $(x\ra y=1$ and $y\ra x=1)$ $\Longrightarrow x=y$ (antisymmetry of $\le$). \\  
Then $X$ is $\vee$-commutative. 
\end{proposition}
\begin{proof} 
Using $(BE_4)$, Lemma \ref{qbe-10}$(6)$ and $(QW_1)$, we get: \\
$\hspace*{1.20cm}$ $(x\vee y)\ra (y\vee x)=(x\vee y)\ra ((y\ra x)\ra x)$ \\
$\hspace*{4.00cm}$ $=(y\ra x)\ra ((x\vee y)\ra x)$ \\
$\hspace*{4.00cm}$ $=(y\ra x)\ra (x^*\ra (x\vee y)^*)$ \\
$\hspace*{4.00cm}$ $=(y\ra x)\ra (x^*\ra (x^*\wedge y^*))$ \\
$\hspace*{4.00cm}$ $=(y\ra x)\ra (x^*\ra y^*)$ \\
$\hspace*{4.00cm}$ $=(y\ra x)\ra (y\ra x)=1$. \\
Similarly $(y\vee x)\ra (x\vee y)=1$, and applying $(An)$ we get $x\vee y=y\vee x$. 
Hence $X$ satisfies condition $($$\vee$-$comm)$. 
\end{proof}

\begin{theorem} \label{qmv-40-30} In any involutive BE algebra, \\
$\hspace*{3cm}$ $(QW)$+$(An)$ $\iff$ $($$\vee$-$comm)$ \\
(i.e. a quantum-Wajsberg algebra is a Wajsberg algebra if and only if $(An)$ (antisymmetry of $\le$) holds).
\end{theorem}
\begin{proof} 
Let $(X,\ra,^*,1)$ be an involutive BE algebra. 
Since $(QW)$ implies $(QW_1)$, by Proposition \ref{qmv-40-20} we have $(QW)$+$(An)$ $\Longrightarrow$ $($$\vee$-$comm)$. 
Conversely, if $X$ is $\vee$-commutative, by Remark \ref{qbe-130}, $X$ is a Wajsberg algebra. 
Hence $\le$ is antisymmetric and, according to Proposition \ref{cqbe-110}, $X$ is a quantum-Wajsberg algebra, hence it satisfies condition $(QW)$. 
\end{proof}

As we can see in the next example, there exist QMV algebras which are not MV algebras. 
 
\begin{example} \label{qmv-50}
Let $(X,\ra,^*,1)$ be the QW algebra from Example \ref{qbe-120} and let $\odot$, $\oplus$ be the operations 
derived from $\ra$ and $^*$.   
\[
\begin{array}{c|ccccccc}
\odot & 0 & a & b & c & 1 \\ \hline
0     & 0 & 0 & 0 & 0 & 0 \\ 
a     & 0 & b & 0 & 0 & a \\ 
b     & 0 & 0 & 0 & 0 & b \\ 
c     & 0 & 0 & 0 & 0 & c \\
1     & 0 & a & b & c & 1  
\end{array}
\hspace{10mm}
\begin{array}{c|ccccccc}
\oplus & 0 & a & b & c & 1 \\ \hline
0      & 0 & a & b & c & 1 \\ 
a      & a & 1 & 1 & 1 & 1 \\ 
b      & b & 1 & a & 1 & 1 \\ 
c      & c & 1 & 1 & 1 & 1 \\
1      & 0 & a & b & c & 1  
\end{array}
\]

Then $(X,\odot,^*,1)$ is a quantum-MV algebra.   
Since $(a^*\odot c)^*\odot c=c\ne a=(c^*\odot a)^*\odot a$, then $X$ is not an MV algebra. 
\end{example}

\begin{example} \label{qmv-60}
Let $(X,\ra,^*,1)$ be the QW algebra from Example \ref{cqbe-50} and let $\odot$, $\oplus$ be the operations 
derived from $\ra$ and $^*$.   
\[
\begin{array}{c|ccccccc}
\odot & 0 & a & b & c & 1 \\ \hline
0     & 0 & 0 & 0 & 0 & 0 \\ 
a     & 0 & c & 0 & b & a \\ 
b     & 0 & 0 & 0 & 0 & b \\ 
c     & 0 & b & 0 & 0 & c \\
1     & 0 & a & b & c & 1  
\end{array}
\hspace{10mm}
\begin{array}{c|ccccccc}
\oplus & 0 & a & b & c & 1 \\ \hline
0      & 0 & a & b & c & 1 \\ 
a      & a & 1 & 1 & 1 & 1 \\ 
b      & b & 1 & c & 1 & 1 \\ 
c      & c & 1 & a & 1 & 1 \\
1      & 0 & a & b & c & 1  
\end{array}
\]

Then $(X,\odot,^*,1)$ is a quantum-MV algebra which is an MV algebra. 
\end{example}

\begin{remark} \label{qmv-65}
We can find in \cite{Ior35} plenty of examples of quantum-MV algebras that are not MV algebras, from which we can obtain by the transformation $\Psi$ examples of quantum-Wajsberg algebras that are not Wajsberg algebras. 
\end{remark}

\begin{proposition} \label{qmv-70}
Let $(X,\ra,^*,1)$ be a QW algebra such that $X$ is a self distributive BE algebra. 
Then the quantum-MV algebra $(X,\odot,^*,1)$ is an MV algebra. 
\end{proposition}
\begin{proof}
According to Proposition \ref{qbe-30}$(4)$, $x\vee y=y\vee x$, for all $x,y\in X$. 
Taking into consideration Remark \ref{pr-20-10}, it follows that $(X,\odot,^*,1)$ is an MV algebra.
\end{proof}

\begin{proposition} \label{qmv-80} 
Let $(X,\odot,^*,1)$ be a quantum-MV algebra. Then, the following hold, for all $x,y,z\in X$: \\ 
$(1)$ $x\odot (x^*\vee_m^M y)=x\odot y$ and $x\odot (y\vee_m^M x^*)=x\odot y;$ \\
$(2)$ $x\odot y\le_m^M x\le_m^M x\oplus y;$ \\
$(3)$ $x\odot y=1$ iff $x=y=1$ and $x\oplus y=0$ iff $x=y=0;$ \\
$(4)$ $x\vee_m^M (x\odot y)=x$ and $x\wedge_m^M (x\oplus y)=x;$ \\
$(5)$ $x\odot (y\vee_m^M (z\vee_m^M x^*))=(x\odot y)\vee_m^M (x\odot z);$ \\   
$(6)$ $x\oplus y\oplus (x\odot y)=x\oplus y$ and $x\odot y\odot (x\oplus y)=x\odot y;$ \\
$(7)$ $(x^*\oplus y)\vee_m^M (y^*\oplus x)=1$ and $(x\odot y^*)\wedge_m^M (y\odot x^*)=0;$ \\
$(8)$ $(z\wedge_m^M x)^*\oplus (y\wedge_m^M x)=(z\wedge_m^M x)^*\oplus y$ and 
$(z\wedge_m^M x)\odot (y\wedge_m^M x)^*=(z\wedge_m^M x)\odot y^*$.  
\end{proposition}
\begin{proof}
$(1)$: By $(QW_1)$, $x\ra (x\wedge_m^M y^*)=x\ra y^*$, so that $x\odot (x^*\vee_m^M y)=x\odot y$. 
Similarly, by Proposition \ref{qbe-80}$(1)$, $x\ra (y^*\wedge_m^M x)=x\ra y^*$, hence 
$x\odot (y\vee_m^M x^*)=x\odot y$. \\
$(2)$: It follows by Proposition \ref{qbe-80}$(2)$. \\
$(3)$: It follows by $(2)$. \\
$(4)$: We apply $(2)$ and Proposition \ref{qbe-20}$(2)$. \\ 
$(5)$: The identity $x\odot (y\vee_m^M (z\vee_m^M x^*))=(x\odot y)\vee_m^M (x\odot z)$ is equivalent to 
$x\ra (y^*\wedge_m^M (z^*\wedge_m^M x))=(x\ra y^*)\wedge_m^M (x\ra z^*)$ and we use $(QW_2)$. \\
$(6)$: We have: \\
$\hspace*{1.00cm}$ $x\oplus y=x^*\ra y=y^*\ra x=(x\ra y^*)\ra (y^*\ra x)$ (Prop. \ref{qbe-100}$(6)$) \\
$\hspace*{1.90cm}$ $=(x\ra y^*)\ra (x^*\ra y)=(x^*\ra y)^*\ra (x\ra y^*)^*$ \\
$\hspace*{1.90cm}$ $=x\oplus y\oplus (x\odot y)$. \\
The identity $x\odot y\odot (x\oplus y)=x\odot y$ is equivalent to $x\oplus y\oplus (x\odot y)=x\oplus y$. \\
$(7)$: By Proposition \ref{qbe-100}$(7)$, $(x\ra y)\vee_m^M (y\ra x)=1$, so that 
$(x^*\oplus y)\vee_m^M (y^*\oplus x)=1$. 
Moreover, $(x\odot y^*)\wedge_m^M (y\odot x^*)=(x\ra y)^*\wedge_m^M (y\ra x)^*=((x\ra y)\vee_m^M (y\ra x))^*=1^*=0$. \\
$(8)$: It follows by Proposition \ref{qbe-100}$(8)$. 
\end{proof}

\begin{proposition} \label{qmv-90} Let $X$ be a quantum-MV algebra. 
If $x,y\in X$ such that $x\le_m^M y$, then, the following hold, for any $z\in X$:\\
$(1)$ $x\oplus z\le_m^M y\oplus z;$ \\
$(2)$ $x\odot z\le_m^M y\odot z$. 
\end{proposition}
\begin{proof}
$(1)$: By Proposition \ref{qbe-90}, $x\le_m^M y$ implies $y^*\le_m^M x^*$ and $x^*\ra z\le_m^M y^*\ra z$. 
It follows that $x\oplus z\le_m^M y\oplus z$. \\ 
$(2)$: Since $y^*\le_m^M x^*$, using $(1)$ we get $y^*\oplus z^*\le_m^M x^*\oplus z^*$, so that 
$(x^*\oplus z^*)^*\le_m^M (y^*\oplus z^*)^*$. Hence $x\odot z\le_m^M y\odot z$.  
\end{proof}

\begin{theorem} \label{qmv-100} 
A quantum-MV algebra $X$ is an MV algebra if and only if, for all $x,y\in X$, $x\le_m y$ implies $x\le_m^M y$. 
\end{theorem}
\begin{proof} 
It follows from Theorems \ref{cqbe-120}, \ref{qmv-30} and Remarks \ref{qmv-40}$(1)$. 
\end{proof}

\begin{theorem} \label{qmv-110} 
If a quantum-MV algebra $(X,\odot,^*,1)$ satisfies the following conditions for all $x,y\in X$: \\
$\hspace*{2cm}$ $x\wedge_m^M (x\vee_m^M y)=x$ and $(y^*\vee_m^M x)\odot y=x\odot y$, \\
then $X$ is an MV algebra.  
\end{theorem}
\begin{proof} 
From $(y^*\vee_m^M x)\odot y=x\odot y$ we get $x\ra y^*=(y^*\vee_m^M x)\ra y^*$, and replacing $y$ by $y^*$, we have  
$x\ra y=(y\vee_m^M x)\ra y$. 
Since $x\wedge_m^M (x\vee_m^M y)=x$ is equivalent to $x\le_m^M x\vee y$, applying Theorem \ref{cqbe-70}, it follows that $(X,\ra,^*,1)$ is a $\vee$-commutative quantum-Wajsberg algebra. 
Hence $X$ is a bounded $\vee$-commutative BCK algebra, so that $(X,\odot,^*,1)$ is an MV algebra.  
\end{proof}

\begin{theorem} \label{qmv-120} 
If a quantum-MV algebra $(X,\odot,^*,1)$ satisfies the following condition for all $x,y,z\in X$: \\
$\hspace*{3cm}$ $x\odot (y\odot z)=(x\odot y^*)^*\odot (x\odot z)$, \\
then $X$ is an MV algebra.  
\end{theorem}
\begin{proof} 
The identity $x\odot (y\odot z)=(x\odot y^*)^*\odot (x\odot z)$ is equivalent to 
$x\ra (y\ra x^*)=(x\ra y)\ra (x\ra z^*)$, and we get 
for all $x,y,z\in X$. 
Replacing $z$ by $z^*$ we have $x\ra (y\ra z)=(x\ra y)\ra (x\ra z)$. 
Hence $(X\ra,0,1)$ is a self distributive BE algebra, so that by Proposition \ref{qbe-30}, $X$ is 
$\vee$-commutative. Thus it is a bounded $\vee$-commutative BCK algebra, so that $(X,\odot,^*,1)$ is an MV algebra.  
\end{proof}

\begin{remark} \label{qmv-130} 
If $(X,\odot,^*,1)$ is an MV algebra, then we define the operations 
$x\wedge y=x\wedge_m^M y=y\wedge_m^M x$, i.e. 
$x\wedge y=(x^*\odot y)^*\odot y=(y^*\odot x)^*\odot x$ and $x\vee y=(x^*\wedge y^*)^*$. 
Then $(X,\wedge,\vee,0,1)$ is the lattice structure of $(X,\odot,^*,1)$. 
\end{remark}

Taking into consideration the term-equivalence between MV algebras and Wajsberg algebras (Remarks \ref{qmv-40}$(1)$), certain properties of MV algebras can be deduced from Propositions \ref{cqbe-60}, \ref{cqbe-90}, \ref{cqbe-100}.

$\vspace*{1mm}$

\section{Concluding remarks}

We defined and investigated the quantum-Wajsberg algebras and the $\vee$-commutative quantum-Wajsberg algebras. 
We proved that any Wajsberg algebra (bounded $\vee$-commutative BCK algebra) is both a quantum-Wajsberg algebra and a commutative quantum-B algebra, and gave two conditions for a quantum-Wajsberg algebra to be a Wajsberg algebra. 
The notions of ideals and filters in quantum-MV algebras have been defined and studied in 
\cite{Giunt1, Giunt7, DvPu}. 
A nonempty subset $I$ of a QMV algebra $X$ is an \emph{ideal} (or a \emph{q-ideal}) of $X$ if it satisfies the 
conditions: $(I_1)$ $x\in I$, $y\in X$ $\Rightarrow$ $x\odot y\in I$, 
$(I_2)$ $x,y\in I$ $\Rightarrow$ $x\oplus y\in I$. 
An ideal $I$ is a \emph{p-ideal} if $x\in I$, $y\in X$ $\Rightarrow$ $x\wedge y\in I$. 
A subset $F\subseteq X$ is a \emph{filter} of $X$ if it satisfies the following conditions: 
$(F_1)$ $x,y\in F$ $\Rightarrow$ $x\odot y\in F$, $(F_2)$ $x\in F$, $x\le_Q y$ $\Rightarrow$ $y\in F$. \\
The notion of a \emph{first meander} of a subset $F\subseteq X$ was defined in \cite{DvPu1} as the set: \\
$\hspace*{3cm}$ $F^1=\{x\in X\mid x\oplus y\in F \Rightarrow y\in F\}$. \\
It was proved in \cite{DvPu1} that the first meander of an ideal in $X$ is a filter, and the first
meander of a filter is an ideal. \\
As another direction of research, one can define and study the ideals and filters in QW algebras, as well as 
the first meander of a subset in QW algebras.

$\vspace*{1mm}$

\section* {\bf\leftline {Funding}}
\noindent The author declares that no funds, grants, or oher support were received during the preparation of 
this manuscript.  


\section* {\bf\leftline {Competing interests}}
\noindent The author has no relevant financial or non-financial interests to disclose.

$\vspace*{2mm}$
          
\begin{center}
\sc Acknowledgement 
\end{center}
The author is very grateful to the anonymous referees for their useful remarks and suggestions on the subject that helped improving the presentation.

$\vspace*{5mm}$

\end{document}